\documentclass[]{article}
\usepackage{graphicx}
\usepackage{graphics}
\usepackage{amssymb,amsmath,amsthm,epsfig}
\newtheorem{theorem}{Theorem}
\newtheorem{lemma}{Lemma}

\newtheorem{proposition}{Proposition}

\newtheorem{remark}{Remark}
\newcommand{\norm}[1]{{\left\|{#1}\right\|}}

\newcounter{reh}
\newcounter{rek}
\vspace{5mm}
\begin{document}
\begin{center}
	{\large {\bf Non-asymptotic behavior and the distribution of the spectrum of the finite Hankel transform operator}}\\
	\vskip 1cm Mourad Boulsane$^a$ {\footnote{
			Corresponding author: Mourad Boulsane, Email: boulsane.mourad@hotmail.fr}}
\end{center}
\vskip 0.5cm {\small
\noindent $^a$ University of Carthage,
	Department of Mathematics, Faculty of Sciences of Bizerte,Jarzouna, 7021, Tunisia.
}

\begin{abstract}
For a fixed reals $c>0$, $a>0$ and $\alpha>-\frac{1}{2}$, the circular prolate spheroidal wave functions (CPSWFs) or 2d-Slepian functions as some authors call it, are the eigenfunctions of the finite Hankel transform operator, denoted by $\mathcal{H}_c^{\alpha}$, which is the integral operator defined on $L^2(0,1)$ with kernel $H_c^{\alpha}(x,y)=\sqrt{cxy}J_{\alpha}(cxy)$. Also, they are the eigenfunctions of the positive, self-adjoint compact integral operator $\mathcal{Q}_c^{\alpha}=c\mathcal{H}_c^{\alpha}\mathcal{H}_c^{\alpha}.$ The CPSWFs play a central role in many applications such as the analysis of 2d-radial signals. Moreover, a renewed interest on the CPSWFs instead of Fourier-Bessel basis is expected to follow from the potential applications in Cryo-EM and that makes them attractive for steerable of principal component analysis(PCA). For this purpose, we give in this paper a precise non-asymptotic estimates for these eigenvalues, within the three main regions of the spectrum of $\mathcal{Q}_c^{\alpha}$ as well as these distributions in $(0,1).$ Moreover, we describe a series expansion of CPSWFs with respect to the generalized Laguerre functions basis of $L^2(0,\infty)$ defined by $\psi_{n,\alpha}^a(x)=\sqrt{2}a^{\alpha+1}x^{\alpha+1/2}e^{-\frac{(ax)^2}{2}}\widetilde{L}_n^{\alpha}(a^2x^2)$, where $\widetilde{L}_n^{\alpha}$ is the normalised Laguerre polynomial.
\end{abstract}
\vspace{5mm}
\section{Introduction}
We recall that for given values of $\alpha>-\frac{1}{2} $ and $c>0$, the circular prolate spheroidal wave functions (CPSWFs), denoted by $\varphi^{\alpha}_{n,c}$, have been discovered and studied in the 1964's since the pionner works by D. Slepian and his co-authors see \cite{Slepian3}. 
They are at first sight the solution of energy maximization problem, then become the different band-limited eigenfunctions of the finite Hankel transform $\mathcal{H}_c^{\alpha}$ defined on $L^2(0,1)$ with kernel $H_c^{\alpha}(x,y)=\sqrt{cxy}J_{\alpha}(cxy)$ where $J_{\alpha}$ is the Bessel function of the first type and order $\alpha>-\frac{1}{2}$, see for example \cite{Karoui-Boulsane}. That is 
\begin{equation}
\mathcal{H}_c^{\alpha}(\varphi^{\alpha}_{n,c})=\mu_{n,\alpha}(c)\varphi^{\alpha}_{n,c}.
\end{equation}
To the operator $\mathcal{H}_c^{\alpha},$ we associate a positive, self-adjoint compact integral operator $\mathcal{Q}_c^{\alpha}=c\mathcal{H}_c^{\alpha}\mathcal{H}_c^{\alpha}$ defined on $L^2(0,1)$  with kernel $K_c^{\alpha}(x,y)=cG_{\alpha}(cx,cy),$ where
\begin{equation} 
G_{\alpha}(x,y)=\begin{cases}\frac{\sqrt{xy}}{x^2-y^2}\left(xJ_{\alpha+1}(x)J_{\alpha}(y)-yJ_{\alpha+1}(y)J_{\alpha}(x)\right)& x\neq y\\ \frac{1}{2}\left((xJ_{\alpha+1}(x))'J_{\alpha}(x)-xJ_{\alpha+1}(x)J_{\alpha}'(x)\right)&  x=y\end{cases}
\end{equation}
We denote by $\lambda_{n}^{\alpha}(c)$ the infinite and countable sequence of the eigenvalue of the operator $\mathcal{Q}_c^{\alpha}$, that is
$\lambda_{n}^{\alpha}(c)=c|\mu_{n}^{\alpha}(c)|^2.$
In his pioneer work \cite{Slepian3}, D. Slepian has shown that the compact integral operator $\mathcal{H}_c^{\alpha}$ commutes with the following differential operator $\mathcal D^{\alpha}_c$ defined on $C^2([0,1])$ by
\begin{equation}
\label{differ_operator1}
\mathcal D^{\alpha}_c (\phi)(x)  = \dfrac{d}{dx} \left[ (1-x^2)\dfrac{d}{dx} \phi(x) \right] + \left( \dfrac{\dfrac{1}{4}-\alpha^2}{x^2}-c^2x^2 \right)\phi(x).
\end{equation}
Hence,  $\varphi^{\alpha}_{n,c}$ is the $n-$th order bounded   eigenfunction of the operator $-\mathcal D^{\alpha}_c,$  associated with the eigenvalue $\chi_{n,\alpha}(c),$ that is
\begin{equation}
\label{differ_operator2}
-\dfrac{d}{dx} \left[ (1-x^2)\dfrac{d}{dx} \varphi^{\alpha}_{n,c}(x) \right] - \left( \dfrac{\dfrac{1}{4}-\alpha^2}{x^2}-c^2x^2 \right)\varphi^{\alpha}_{n,c}(x)=\chi_{n,\alpha}(c)\varphi^{\alpha}_{n,c}(x),\quad x\in [0,1].
\end{equation}
The 2d-Slepian family form an orthonormal basis of $L^2(0,1).$ 
Moreover, they form an orthogonal basis of the Hankel Paley-Wiener space $\mathcal B_c^{\alpha}$, the space of functions from $L^2(0,\infty)$ with Hankel transforms supported on $[0,c]$, 
\begin{equation}\label{eq0.5}
\mathcal H^{\alpha}(\varphi_{n,c}^{\alpha})=\frac{1}{c\mu_{n,\alpha}(c)}\varphi_{n,c}^{\alpha}(\frac{.}{c})\chi_{[0,c]}
\end{equation}
where $\mathcal H^{\alpha}$ is the Hankel transform defined on $L^2(0,\infty)$ with kernel $H^{\alpha}(x,y)=\sqrt{xy}J_{\alpha}(xy).$\\
In this work, we will first study the decay rate of the eigenvalues of the operator
 $\mathcal{Q}_c^{\alpha}$  which is the heartbeat of this paper. In addition,
  the eigenvalues decreases slowly around the value  
   $\kappa=\frac{c}{\pi}-\frac{\alpha}{2}$, but from $\kappa(1+\epsilon),$
  the decay becomes of exponential type then becomes super-exponential. Moreover, we will describe the results of the eigenvalues decay rates in a non asymptotic way in the three main regions of the $\mathcal{Q}_c^{\alpha}$ spectrum . In their seminal work, Abreu and Bandeira \cite{Abreu} have given the following asymptotic estimation of the trace and Hilbert-Schmidt norm of the operator $\mathcal{Q}_c^{\alpha}$  $$\mbox{trace}(\mathcal{Q}_c^{\alpha})=\frac{c}{\pi}+O(1) \hspace{5mm} \mbox{and} \hspace{5mm} ||\mathcal{Q}_c^{\alpha}||_{HS}\geq \frac{c}{\pi}-K\ln(c)-L$$ for some unknown constants $K$ and $L$. In this paper, we are interested in the non-asymptotic estimation of the trace and norm of $\mathcal{Q}_c^{\alpha}$, we improve the previous results by giving the exact values of K and L, that is 
$$\mbox{trace}(\mathcal{Q}_c^{\alpha})=\frac{c}{\pi}-\frac{\alpha}{2}+\frac{\gamma_{\alpha}(c)}{c}\hspace{5mm} \mbox{and} \hspace{5mm}||\mathcal{Q}_c^{\alpha}||_{HS}\geq \frac{c}{\pi}-\frac{\alpha}{2}-K_1\ln(\frac{c}{\pi})-\frac{K_2}{c^2}-L$$  where $\gamma_{\alpha}(c)$ is bounded (see Theorem 3) and  $K_1$, $K_2$ and $L$ are given by Theorem 4. The last result gives us a good idea about the distribution of the eigenvalues in the interval (0,1). Indeed, for any $\epsilon>0$ we have
\begin{equation}
\#\{n ; \epsilon<\lambda_n^{\alpha}(c)<1-\epsilon \}\leq \frac{\mbox{trace}(\mathcal{Q}_c^{\alpha})-||\mathcal{Q}_c^{\alpha}||_{HS}}{\epsilon(1-\epsilon)}\leq \frac{K_1\ln(\frac{c}{\pi})+\frac{K_2}{c^2}+L}{\epsilon(1-\epsilon)}.
\end{equation}
 
The classical Bouwkamp algorithm is the Slepian scheme for the computation of CPSWFs. The latter expands in $\widetilde{T}_{k,\alpha}$ basis of $L^2(0,1)$ defined by $\widetilde{T}_{k,\alpha}(x)=(-1)^k\sqrt{2(2k+\alpha+1)}x^{\alpha+\frac{1}{2}}P_k^{(\alpha,0)}(1-2x^2)$ where the series expansion coefficients check an eigensystem whose matrix becomes diagonally dominant in the case where the order of the prolate is large compared to the band-limit. In the opposite case, this classical method encounters numerical difficulties. Then,  for more analytical informations, we will give a serie expansion of circular prolate in Laguerre functions basis of $L^2(0,\infty)$ defined by $\psi_{n,\alpha}^a(x)=\sqrt{2}a^{\alpha+1}x^{\alpha+1/2}e^{-\frac{(ax)^2}{2}}\widetilde{L}_n^{\alpha}(a^2x^2)$, where $\widetilde{L}_n^{\alpha}$ is the normalised Laguerre polynomial, wich will become the solution of this problem already posed by Xiao and Rokhlin, see \cite{Xiao-Rokhlin}, within the framework of the classical prolate. However, the use of PSWFs has been somewhat crippled by their slightly mysterious reputation as being ”difficult to compute”. This seems to be related to the fact that the classical (”Bouwkamp”) algorithm for their evaluation encounters numerical difficulties for $c > 40$ or so. Moreover, the attempt to diagonalize the finite Fourier transform operator $\mathcal{F}_c$ numerically via straightforward discretization meets with numerical difficulties as well. In addition, the serie expansion coefficients of the classical PSWFs in the Legendre polynomials basis form an eigenvector of a tri-diagonal matrix that becomes diagonally dominant when the order of the function is large compared to the band-limit and in this case we get an efficient approximate path to calculate the classical PSWFs,  but the problem is posed in the case where the order of the function is less than the length of the band, that is what explained and solved Xiao and Rokhlin from the serie expansion of the PSWFs in the Hermites functions basis of $L^2(\mathbb{R})$ defined by $\phi_n^a(x)=e^{-a^2x^2/2}H_n^a(x)$ where $H_n^a$ is the normalised Hermite polynomial,  where the coefficients form an eigenvector of a five-diagonal matrix that becomes diagonally dominant when the band-limit is large compared to  prolate's order.

This work is organised as follows. In section 2, we give some mathematical preliminaries related to the properties and computation of Bessel functions and generalized Laguerre functions. In section 3, we give the behavior of the eigenvalues decay rates $\lambda_{n}^{\alpha}(c)$ during the interval $(0,1).$ The previous results are obtained according to the position of the prolate's order   in relation to the band-limit. The section 4 is devoted to the precise non-asymptotic estimate of the trace and Hilbert-Schmidt norm of the $\mathcal{Q}_c^{\alpha}$ operator which will give us an idea about the eigenvalues distribution in inteval $(0,1)$. Finally, in section 5, we will describe the serie expansion coefficients of the CPSWFs in generalized Laguerre functions.

\section{Mathematical preliminaries}
In this section, we first give a brief description of Laguerre polynomials and non-asymptotic expansion of Bessel functions. These functions are probably among the most frequently used special functions. Broadly speaking, they occur in connection with Sturm-liouville differential equations or in connection with certain definite integrals. \\
For $\alpha>-\frac{1}{2}$, the Bessel's differential equation is given by, see for example \cite{Watson},

\begin{equation*}
x^2y''+xy'+(x^2-\alpha^2)y=0,\quad x>0,
\end{equation*}
$J_{\alpha}$ is the bounded solution of the previous differential equation of the first kind and order $\alpha$. It has the series expansion  :
\begin{equation*}\label{serie}
J_{\alpha}(x)=\sum_{n=0}^{\infty}\frac{(-1)^n}{n!}\frac{(x/2)^{2n+\alpha}}{\Gamma(n+\alpha+1)}.
\end{equation*}
Bounds and local estimates of $J_{\alpha}$ are frequently needed in this work. For a comprehensive review
of these bounds and estimates, the reader is referred to [\cite{Ilia Krasikov}, \cite{A.YA. OLenko}, \cite{Watson}]. A first simple and useful local estimate is
given by, see \cite{A.YA. OLenko}
\begin{equation}
\label{boundJ}
\sup_{x\geq 0} \sqrt{x} |J_{\alpha}(x)| \leq C_{\alpha}(O),
\end{equation}
where \begin{equation*}
\label{constants}
C_{\alpha}(O) =\left\{\begin{array}{ll}
\sqrt{2/\pi} &\mbox{ if } |\alpha|\leq 1/2\\
0.675\sqrt{\alpha^{1/3}+\frac{1.9}{\alpha^{1/3}}+
	\frac{1.1}{\alpha}}&\mbox{ if } \alpha >1/2.\end{array}\right.
\end{equation*}
In \cite{A.YA. OLenko}, Olenko pushed the estimate of Bessel's function as follows
\begin{equation}
\sup_{x\geq 0} x^{3/2}|J_{\alpha}(x)-\sqrt{\frac{2}{\pi x}}\mbox{cos}\left(x-(\frac{\alpha\pi}{2}+\frac{\pi}{4})\right)| \leq d_{\alpha}(O)=2.2b\left(50\sqrt{2\pi}+\frac{3}{13}\left((\alpha+1)^{13/6}-\frac{1}{2^{13/6}}\right)\right)
\end{equation}
where $b=0.674885...$\\
It must be said that the inequality of Olenko is interesting, especially it brings us back to an advanced uniform approximation of Bessel's function. Using the same technique, we can further improve its non-asymptotic developpment and we obtain
\begin{equation}
\sup_{x\geq 0}x^{5/2}|J_{\mu}(x)-\sqrt{\frac{2}{\pi x}}\left(\cos(x-\frac{\mu\pi}{2}-\frac{\pi}{4})-(\mu^2-\frac{1}{4})\frac{\sin(x-\frac{\mu\pi}{2}-\frac{\pi}{4})}{2x}\right)|\leq d_{\mu}\\
\end{equation}
where $d_{\mu}\leq 196\sqrt{\frac{2}{\pi}}+2C(\mu-1)(\mu-2)^{13/6}(\frac{\mu}{2}+\frac{1}{4}),$ with  $C=\left(4.4b+50\sqrt{2\pi}2^{13/6}+\frac{3}{13}\right).$
Note that in \cite{Ilia Krasikov}, Krasikov has improved the Olenko constant as follows
\begin{equation}
\label{boundJ_1}
\sup_{x\geq 0} x^{3/2}|J_{\alpha}(x)-\sqrt{\frac{2}{\pi x}}\cos\left(x-(\frac{\alpha\pi}{2}+\frac{\pi}{4})\right)| \leq C_{\alpha}(K)=\frac{4}{5}|\alpha^2-\frac{1}{4}|.
\end{equation}
The following proposition improves the Krasikov result .

\begin{proposition}
	Let $\mu>-1/2$, then we have
	\begin{equation}\label{BoundJ_2}
	\sup_{x\geq 0}x^{5/2}|J_{\mu}(x)-\sqrt{\frac{2}{\pi x}}\left(\cos(x-\frac{\mu\pi}{2}-\frac{\pi}{4})-(\mu^2-\frac{1}{4})\frac{\sin(x-\frac{\mu\pi}{2}-\frac{\pi}{4})}{2x}\right)|\leq \frac{\beta}{\pi}\left(\frac{4}{5}\beta+\sqrt{\frac{2}{\pi}}\right)=C_{\mu}
	\end{equation}
	where $\beta=|\mu^2-1/4|.$
\end{proposition}
\begin{proof}
	Let
	\begin{eqnarray*} 
		K(x)&=&\sqrt{\frac{\pi x}{2}}J_{\mu}(x)-\mbox{cos}(x-\omega_{\mu})+(\mu^2-1/4)\frac{\mbox{sin}(x-\omega_{\mu})}{2x}\\
		&=& r(x)+(\mu^2-1/4)\frac{\mbox{sin}(x-\omega_{\mu})}{2x},
	\end{eqnarray*}
	where $\omega_{\mu}=\frac{\mu\pi}{2}+\frac{\pi}{4}.$ Then $K$ satisfies the following differential equation
	\begin{eqnarray*}
		K''(x)+K(x)&=&r''(x)+r(x)+(\mu^2-1/4)\frac{\mbox{sin}(x-\omega_{\mu})-x\mbox{cos}(x-\omega_{\mu})}{x^3}\\
		&=&\sqrt{\frac{\pi}{2x^3}}(\mu^2-1/4)J_{\mu}(x)+(\mu^2-1/4)\frac{\mbox{sin}(x-\omega_{\mu})-x\mbox{cos}(x-\omega_{\mu})}{x^3}.
	\end{eqnarray*}
	We know from \cite{Ilia Krasikov} that the solution of the differential equation 
	$$K_1''(x)+K_1(x)=\sqrt{\frac{\pi}{2x^3}}(\mu^2-1/4)J_{\mu}(x)$$ is given by
	\begin{equation}
	K_1(x)=\sqrt{\frac{\pi}{2}}(\mu^2-1/4)\int_x^{\infty}\frac{\mbox{sin}(t-x)}{t^{3/2}}J_{\mu}(t)dt.
	\end{equation}
	We can easly prove that the solution of the last differential equation is given by
	$$K_2''(x)+K_2(x)=(\mu^2-1/4)\frac{\mbox{sin}(x-\omega_{\mu})-x\mbox{cos}(x-\omega_{\mu})}{x^3}.$$

	\begin{equation}
	K_2(x)=(\mu^2-1/4)\int_x^{\infty}\mbox{sin}(t-x)\frac{\mbox{sin}(t-\omega_{\mu})-t\mbox{cos}(t-\omega_{\mu})}{t^3}dt.
	\end{equation}
	We conclude that
	\begin{equation}
	K(x)=K_1(x)+K_2(x)=\sqrt{\frac{\pi}{2}}(\mu^2-1/4)\int_x^{\infty}\frac{\mbox{sin}(t-x)}{t^{5/2}}\left(tJ_{\mu}(t)+\sqrt{\frac{2}{\pi}}h_{\mu}(t)\right)dt,
	\end{equation}
	where $h_{\mu}(t)=\displaystyle\frac{\mbox{sin}(t-\omega_{\mu})-t\mbox{cos}(t-\omega_{\mu})}{\sqrt{t}}.$ For the quantity
	\begin{equation*}
	\sqrt{t}\left(tJ_{\mu}(t)+\sqrt{\frac{2}{\pi}}h_{\mu}(t)\right)=t^{3/2}\left(J_{\mu}(t)-\sqrt{\frac{2}{\pi t}}\mbox{cos}(t-\omega_{\mu})\right)+\sqrt{\frac{2}{\pi}}\mbox{sin}(t-\omega_{\mu}),
	\end{equation*}
	we have from \cite{Ilia Krasikov}, $$\sqrt{t}\left|tJ_{\mu}(t)+\sqrt{\frac{2}{\pi}}h_{\mu}(t)\right|\leq \frac{4}{5}|\mu^2-1/4|+\sqrt{\frac{2}{\pi}}$$
	Using the last inequality and the following result, for a non-negative and decreasing function $f\in L^1(0,\infty)$ 
	$$\int_0^{\infty}f(t)|\mbox{sin}(t)|dt\leq \frac{2}{\pi}\int_0^{\infty}f(t)dt $$ 
	one gets
	\begin{eqnarray} 
	|K(x)|&=&|\sqrt{\frac{\pi x}{2}}J_{\mu}(x)-\left(\mbox{cos}(x-\omega_{\mu})-(\mu^2-1/4)\frac{\mbox{sin}(x-\omega_{\mu})}{2x}\right)|\\
	&\leq&\sqrt{\frac{2}{\pi}}\frac{\left(\frac{4}{5}|\mu^2-1/4|+\sqrt{\frac{2}{\pi}}\right)|\mu^2-1/4|}{2x^2}.
	\end{eqnarray}
	This concludes the proof of the proposition.
\end{proof}

 Next, the generalised Laguerre polynomials $(L_n^{\alpha})$ associated with non-negative integer $n$ and real number $\alpha>-1$ are solutions of the following second-order linear differential equation:
 \begin{equation}\label{diff,Lag}
 xy''+(\alpha+1-x)y'+ny=0.
 \end{equation}
  The Laguerrepolynomials satisfy the following recurrence formula, for every $x\geq 0$
  \begin{equation}\label{recu.Lag}
\begin{cases}L_0^{\alpha}(x)=1\\ L_{n+1}^{\alpha}(x)=\frac{2n+1+\alpha-x}{n+1}L_n^{\alpha}(x)-\frac{n+\alpha}{n+1}L_{n-1}^{\alpha}(x) \hspace{5mm} n\geq 1\end{cases}
 \end{equation}
Also,they are given by the Rodriguez formula  $$L_n^{\alpha}(x)=\displaystyle\frac{x^{-\alpha}e^x}{n!}\frac{d^n}{dx^n}(e^{-x}x^{n+\alpha}).$$ Moreover, the family $\left\{\widetilde{L}_n^{\alpha}=\sqrt{\frac{n!}{\Gamma(n+\alpha+1)}}L_n^{\alpha}\right\}$ constitues a complete orthonormal system in $L^2((0,\infty),\omega_{\alpha})$ where $\omega_{\alpha}(x)=x^{\alpha}e^{-x}$, that is
\begin{equation}
\int_0^{\infty}L_n^{\alpha}(x)L_m^{\alpha}(x)x^{\alpha}e^{-x}dx=\frac{\Gamma(n+\alpha+1)}{n!}\delta_{nm}.
\end{equation}
Consider the family of generalised Laguerre functions defined by $$\psi_{n,\alpha}^{a}(x)=\sqrt{2}a^{\alpha+1}x^{\alpha+1/2}e^{-\frac{(ax)^2}{2}}\widetilde{L}_n^{\alpha}(a^2x^2).$$\\
The family $(\psi_{n,\alpha}^a)$ is an orthonormal basis of $L^2(0,\infty)$. For $a=1$, the latter becomes the sequence of eigenfunctions of the Hankel transform defined in $L^2(0,\infty)$ by
\begin{equation}
\mathcal{H}^{\alpha}(f)(x)=\int_0^{\infty}\sqrt{xy}J_{\alpha}(xy)f(y)dy,\end{equation} see for example \cite{Andrews}, that is 
\begin{equation}\label{eq0.3} \mathcal{H}^{\alpha}(\psi_{n,\alpha}^a)=\frac{(-1)^n}{a}\psi_{n,\alpha}^a(\frac{.}{a^2}).
\end{equation}

\section{Estimates of the eigenvalues.}
In this paragraph, we extend to the case of the finite Hankel transform the recent non-asymptotic results concerning the behaviour of the spectrum of the Sinc-kernel operator. In particular, we show how the $\lambda_{n}^{\alpha}(c)$ converge to 1. Also, we give a decay rate of the $\lambda_{n}^{\alpha}(c)$ near the plunge region around $\frac{c}{\pi}.$ Finally, we give a super-exponential decay rate for the $\lambda_{n}^{\alpha}(c).$ The following lemma is needed in the proof of our first result concerning the behavior of the $\lambda_{n}^{\alpha}(c).$
	\begin{lemma}
		\label{lemma 0.1}
		Let $\alpha\geq -1/2$ and $(L_{n}^{\alpha})$  be the sequence of the laguerre polynomials of order $\alpha.$ 
		Then for every $n\geq 0$ we have
		\begin{equation} 
		|L_{n}^{\alpha}(x)|\leq \frac{x^n}{n!} \hspace{5mm}\forall\hspace{0.5mm} x\geq 4n+C_{\alpha}
		\end{equation}
		where $C_{\alpha}=
		\begin{cases} 2(\alpha+1)&\mbox{if}~~\displaystyle\alpha\geq 1/2\\
		2(\alpha+1)+1/16&\mbox{if}\displaystyle -1/2\leq \alpha<1/2.
		\end{cases}
		$
	\end{lemma}
	\begin{proof}
		Let $\gamma_{n}^{\alpha}$ be the largest zero of the Laguerre polynomial $L_{n}^{\alpha}$, then for every $x\geq \gamma_{n}^{\alpha}$ we have
		$L_{n}^{\alpha}(x)=\displaystyle\int_{\gamma_{n}^{\alpha}}^x\frac{dL_{n}^{\alpha}}{dt}(t_1)dt_1.$ By \cite{Andrews}, we have  $\displaystyle\frac{dL_{n}^{\alpha}}{dt}=-L_{n-1}^{\alpha+1}$, then we obtain
		\begin{eqnarray*}
			L_{n}^{\alpha}(x)&=&(-1)^n\int_{\gamma_{n}^{\alpha}}^x\int_{\gamma_{n-1}^{\alpha+1}}^{t_1}\int_{\gamma_{n-2}^{\alpha+2}}^{t_2}...\int_{\gamma_{0}^{\alpha+n}}^{t_n}L_{0}^{\alpha+n}(t_n)dt_ndt_{n-1}...dt_1\\&=&(-1)^n\int_{\gamma_{n}^{\alpha}}^x\int_{\gamma_{n-1}^{\alpha+1}}^{t_1}\int_{\gamma_{n-2}^{\alpha+2}}^{t_2}...\int_{\gamma_{0}^{\alpha+n}}^{t_n}1dt_ndt_{n-1}...dt_1.
		\end{eqnarray*}
		Consequently, for every $x\geq \gamma_{n}^{\alpha},$ we have
		\begin{eqnarray*}
			|L_{n}^{\alpha}(x)|&\leq& \int_{\gamma_{n}^{\alpha}}^x\int_{\gamma_{n-1}^{\alpha+1}}^{t_1}\int_{\gamma_{n-2}^{\alpha+2}}^{t_2}...\int_{\gamma_{0}^{\alpha+n}}^{t_n}dt_ndt_{n-1}...dt_1.\\
			&\leq& \int_0^x\int_0^{t_1}\int_0^{t_2}...\int_0^{t_n}dt_ndt_{n-1}...dt_1.=\frac{x^n}{n!}.
		\end{eqnarray*}
		By \cite{L.G}, we have an upper bound of $\gamma_{n}^{\alpha}$ given by 
		\begin{eqnarray*}
			\gamma_{n}^{\alpha}&\leq& \delta_n^{\alpha}:=2n+\alpha+1+\left[(2n+\alpha+1)^2+1/4-\alpha^2\right]^{1/2}\\
			&\leq& 4n+C_{\alpha}.
		\end{eqnarray*}
	Hence, for every $x\geq 4n+ C_{\alpha},$ we have $|L_{n}^{\alpha}(x)|\leq \frac{x^n}{n!}.$	
	\end{proof}
The following theorem and its proof are largely inspired and follows the same lines of proof of similar theorem given in \cite{BJK} in the case of the eigenvalues of the Sinc-kernel operator.
	\begin{theorem}
		Let $\alpha\geq -1/2$, then under the above notation, we have, for every $c>C_{\alpha}$ and $n<\frac{c-C_{\alpha}}{4}$
		\begin{equation}
		\lambda_n^{\alpha}(c)\geq 1-\left(\frac{20}{3}\right)c^{\alpha}\frac{c^{2n}}{n!}e^{-c}.
		\end{equation}
		
	\end{theorem}
	
	\begin{proof}
	Let $(\psi_n^{\alpha})$ be the sequence of eigenfunctions of the Hankel transform defined in \eqref{eq0.3}, by the previous lemma, we have
		$$|\psi_n^{\alpha}(x)|\leq\frac{ C_n^{\alpha}}{n!}x^{2n+\alpha+1/2}e^{-x^2/2},\hspace{5mm}\forall\hspace{0.5mm} x>(4n+C_{\alpha})^{1/2}$$
		Using the following inequality valid for $\beta>1$ and $a>\sqrt{\beta-1}$
		$$\int_a^{\infty}t^{\beta}e^{-t^2}dt\leq a^{\beta-1}e^{-a^2},$$
    	we deduce that for every $k<\min\left\{\frac{a^2-2\alpha}{4},\frac{a^2-C_{\alpha}}{4}\right\},$ we have 
		\begin{equation}\label{I}
		\int_a^{\infty}|\psi_k^{\alpha}(t)|^2dt\leq (\frac{ C_k^{\alpha}}{(k)!})^2\int_a^{\infty}t^{4k+2\alpha+1}e^{-t^2}dt\leq ( \frac{C_k^{\alpha}}{k!})^2 a^{4k+2\alpha}e^{-a^2}.
		\end{equation}
		Hence, by using the fact that $\int_0^{\infty}|\psi_k^{\alpha}(t)|^2dt=1$, one gets
		\begin{equation}
		\int_0^a|\psi_k^{\alpha}(t)|^2dt\geq 1- \left( \frac{C_k^{\alpha}}{k!}\right)^2 a^{4k+2\alpha}e^{-a^2}.
		\end{equation}
		
		Next, we define $\psi_{n,c}^{\alpha}(t)=c^{1/4}\psi_n^{\alpha}(\sqrt{c} t)$, this last family is also an orthonormal basis of $L^2(0,\infty)$. Let $V_n(c)=\mbox{Span}\displaystyle\{P_c\psi_{0,c}^{\alpha},...,P_c\psi_{n,c}^{\alpha}\}$ with $P_c$ is the orthonormal projection on $B_c^{\alpha}$ defined by
		$$B_c^{\alpha}=\left\{f\in L^2(0,\infty),\mbox{Support}\left(\mathcal{H}^{\alpha}(f)\right)\subset[0,c]\right\}$$
		 From min-max theorem, we know that 
		\begin{eqnarray*}
			\lambda_n^{\alpha}(c)&\geq& \mbox{inf}\left\{\frac{\norm f_{L^2[0,1]}^2}{\norm f_{L^2(0,\infty)}^2}; f\in V_n(c)\right\}\\
			&\geq& \mbox{inf}\left\{\frac{\norm f_{L^2[0,1]}^2}{\norm f_{L^2(0,\infty)}^2}; f=\displaystyle\sum_{j=0}^n\gamma_jP_c\psi_{j,c}^{\alpha}, \sum_{j=0}^n|\gamma_j|^2=1\right\}
		\end{eqnarray*}
		
		Let $F=\displaystyle\sum_{j=0}^n\gamma_j\psi_{j,c}^{\alpha}$, $f=P_cF$ and $\widetilde{P}_c=I-P_c$, then we have
		\begin{eqnarray*}
			\norm f_{L^2[0,1]}^2&=&\left\|P_cF\right\|_{L^2[0,1]}^2=\left\|F-\widetilde{P}_cF\right\|_{L^2[0,1]}^2\\
			&\geq&\left\|F-\widetilde{P}_cF\right\|_{L^2(0,\infty)}^2-\left\|F-\widetilde{P}_cF\right\|_{L^2(1,\infty)}^2\\
			&\geq&\left\|F\right\|_{L^2(0,\infty)}^2-\left\|\widetilde{P}_cF\right\|_{L^2(0,\infty)}^2-\left\|F-\widetilde{P}_cF\right\|_{L^2(1,\infty)}^2\\
			&\geq&1-3\left\|\widetilde{P}_cF\right\|_{L^2(0,\infty)}^2-2\left\|F\right\|_{L^2(1,\infty)}^2.
		\end{eqnarray*}
		We estimate the two errors terms $\left\|\widetilde{P}_cF\right\|_{L^2(0,\infty)}^2$ and $\left\|F\right\|_{L^2(1,\infty)}^2$ and we will start with the last one.\\
		Using the Cauchy-Schwarz inequality and the fact that $\sum_{j=0}^n|\gamma_j|^2=1$, we have
		\begin{eqnarray*}
			\left\|F\right\|_{L^2(1,\infty)}^2&=&\int_1^{\infty}\left|\displaystyle\sum_{j=0}^n\gamma_j\psi_{j,c}^{\alpha}(t)\right|^2dt
			=\int_{\sqrt c}^{\infty}\left|\displaystyle\sum_{j=0}^n\gamma_j\psi_{j}^{\alpha}(t)\right|^2dt
			\leq\displaystyle\sum_{j=0}^n\int_{\sqrt c}^{\infty}\left|\psi_{j}^{\alpha}(t)\right|^2dt.
		\end{eqnarray*}
		Using \eqref{I}, one gets
		$$\left\|F\right\|_{L^2(1,\infty)}^2\leq  \displaystyle\sum_{j=0}^n(\frac{C_j^{\alpha}}{j!})^2 c^{2j+\alpha}e^{-c}\leq\displaystyle\sum_{j=0}^n\frac{ c^{2j+\alpha}}{j!}e^{-c}$$
		
		Since $n<\frac{c-C_{\alpha}}{4}<\frac{c}{4}$ and $c>C_{\alpha}\geq 1$, then we have
		$$\displaystyle\sum_{j=0}^n\frac{ c^{2j}}{j!}\leq\frac{c^{2n}}{n!}\sum_{j=0}^n\left(\frac{n}{c^2}\right)^j\leq
		4/3\frac{c^{2n}}{n!}.$$
		Then, we obtain
		\begin{equation}\label{0.3}
		\left\|F\right\|_{L^2(1,\infty)}^2\leq (4/3)c^{\alpha}\frac{c^{2n}}{n!}e^{-c}.
		\end{equation}
		Let's give now an upper bound of $\left\|\widetilde{P}_cF\right\|_{L^2(0,\infty)}^2.$
		As we have $$\left(B_c^{\alpha}\right)^{\bot}=\left\{f\in L^2(0,\infty), \mbox{Support}\left(\mathcal{H}_{\alpha}(f)\right)\subset (c,\infty)\right\}$$
		then we have
		$$\left\|\widetilde{P}_cF\right\|_{L^2(0,\infty)}^2=\left\|\mathcal{H}^{\alpha}(\widetilde{P}_cF)\right\|_{L^2(0,\infty)}^2=\int_c^{\infty}\left|\mathcal{H}^{\alpha}(F)(x)\right|^2dx.$$
		From \eqref{eq0.3}, we have
		$$\mathcal{H}^{\alpha}(F)(x)=\displaystyle\sum_{j=0}^n\gamma_j\mathcal{H}^{\alpha}(\psi_{j,c}^{\alpha})(x)=\displaystyle\sum_{j=0}^nc^{-1/4}(-1)^j\gamma_j\psi_{j}^{\alpha}(c^{-1/2}x)$$
		By Cauchy-Schwars inequality, the fact that $\sum_{j=0}^n|\gamma_j|^2=1$ and the same techniques used for $\left\|F\right\|_{L^2(1,\infty)}^2$, we have
		\begin{eqnarray*}
			\left\|\mathcal{H}^{\alpha}(\widetilde{P}_cF)\right\|_{L^2(0,\infty)}^2&=&\int_c^{\infty}\left|\displaystyle\sum_{j=0}^nc^{-1/4}(-1)^j\gamma_j\psi_{j}^{\alpha}(c^{-1/2}x)\right|^2dx\\&=&\int_{\sqrt c}^{\infty}\left|\displaystyle\sum_{j=0}^n(-1)^j\gamma_j\psi_{j}^{\alpha}(x)\right|^2dx\\&\leq&\displaystyle\sum_{j=0}^n\int_{\sqrt c}^{\infty}\left|\psi_{j}^{\alpha}(x)\right|^2dx\leq (4/3)c^{\alpha}\frac{c^{2n}}{n!}e^{-c}.
		\end{eqnarray*}
		Hence, we have
		\begin{equation}\label{0.4}
		\left\|\widetilde{P}_cF\right\|_{L^2(0,\infty)}^2=\left\|\mathcal{H}^{\alpha}(\widetilde{P}_cF)\right\|_{L^2(0,\infty)}^2\leq (4/3)c^{\alpha}\frac{c^{2n}}{n!}e^{-c}.
		\end{equation}
		
		Finally, as $\norm f_{L^2(0,\infty)}\leq1$ and by combining \eqref{0.3},\eqref{0.4}, we get our result.
		
	\end{proof}

Numerical evidences indicate that a plunge region for the spectrum $\{\lambda_{n}^{\alpha}(c), n\geq 0\}$ occur around the value $\frac{c}{\pi}-\frac{\alpha}{2}$. By using the non-asymptotic decay rate for the classical $\lambda_{n}(c)$, given by \cite{BJK}, we have, for every $c>22,$ $\eta\geq 0.069$ and $\frac{2c}{\pi}+\ln(c)+6\leq n \leq c,$
\begin{equation}\label{*}
\lambda_{n}(c)\leq \frac{1}{2}\mbox{exp}\left(-\frac{\eta(n-\frac{2c}{\pi})}{\ln(c)+5}\right).
\end{equation}
Also, from \cite{Karoui-Boulsane}, we have, for every $\alpha\geq\alpha'>-\frac{1}{2},$  $\lambda_{n}^{\alpha}(c)\leq\lambda_{n}^{\alpha'}(c).$ In particular, for every $\alpha\geq\frac{1}{2},$ we have
\begin{equation}\label{**}
\lambda_{n}^{\alpha}(c)\leq\lambda_{n}^{\frac{1}{2}}(c)=\lambda_{2n+1}(c).
\end{equation}
By using \eqref{*} and \eqref{**}, we conclude that
	For every $\alpha\geq 1/2$, $c>22$, $\eta\geq 0.069$ and $\frac{c}{\pi}+\frac{\ln(c)}{2}+\frac{5}{2}\leq n \leq \frac{c}{2}-\frac{1}{2},$ we have
	\begin{equation}
\lambda_{n}^{\alpha}(c)\leq	\frac{1}{2}\mbox{exp}\left(-\frac{\eta(2n+1-\frac{2c}{\pi})}{\ln(c)+5}\right).
	\end{equation}
The following theorem gives a super-exponential decay rate for the eigenvalues $\lambda_{n}^{\alpha}(c).$	
	\begin{theorem}
		Let $\alpha\geq -1/2$ and $(\lambda_{n}^{\alpha}(c))$ be the infinite and countable sequence  eigenvalues of the operator $\mathcal{Q}_c^{\alpha}$.  Then there exists a constant $A$ depending only on $c$ such that for every $n> \frac{ec}{4},$ we have
		$$\lambda_{n}^{\alpha}(c)\leq A\left(\frac{ec}{4n+2\alpha+5}\right)^{2n+\alpha+1}.$$
	\end{theorem}
	\begin{proof}
		Let $(T_k^{\alpha})$ be the family of the functions of $L^2[0,1]$ defined by $$T_k^{\alpha}(x)=(-1)^k\sqrt{2(2k+\alpha+1)}x^{\alpha+1/2}P_k^{(\alpha,0)}(1-2x^2)\hspace{5mm} k\geq 0,$$ where $P_k^{(\alpha,0)}$ is the Jacobi polynomials of order $(\alpha,0)$. This family is an orthonormal basis of $L^2[0,1]$.\\	
		By the min-max theorem, we have
		\begin{eqnarray*}
		\lambda_{n}^{\alpha}(c)&=&\displaystyle\min_{\{S_n\}}\max_{\left\{f\in S_n^{\bot},||f||_{L^2[0,1]}=1\right\}}<\mathcal{Q}_c^{\alpha}(f),f>_{L^2[0,1]}\\&=&c\displaystyle\min_{\{S_n\}}\max_{\left\{f\in S_n^{\bot},||f||_{L^2[0,1]}=1\right\}}||\mathcal{H}_c^{\alpha}(f)||^2_{L^2[0,1]},
		\end{eqnarray*}
		where $S_n$ is a subspace of $L^2[0,1]$ of dimension n.\\
		Let $S_n=span\left\{T_0^{\alpha},T_1^{\alpha},...,T_{n-1}^{\alpha}\right\}$, then for every $f\in S_n^{\bot}$ with $||f||_{L^2[0,1]}=1$, we have
		$$f(x)=\sum_{k\geq n}a_kT_k^{\alpha}(x) ~~~~~~~~\forall x\in [0,1].$$
		Hence, we have
		\begin{eqnarray*}
			||\mathcal{H}_c^{\alpha}(f)||_{L^2[0,1]}^2&=&<\mathcal{H}_c^{\alpha}(f),\mathcal{H}_c^{\alpha}(f)>_{L^2[0,1]}\\
			&=&\sum_{k\geq n, j\geq n}a_ka_j<\mathcal{H}_c^{\alpha}(T_k^{\alpha}),\mathcal{H}_c^{\alpha}(T_j^{\alpha})>_{L^2[0,1]}\\
			&=&\sum_{k\geq n}a_k^2<j_k^{\alpha},j_k^{\alpha}>_{L^2[0,1]}\\
			&=&\sum_{k\geq n}a_k^2||j_k^{\alpha}||_{L^2[0,1]}^2.
		\end{eqnarray*}
		Here, note that $j_k^{\alpha}$ is the Spherical Bessel functions defined by $j_k^{\alpha}(x)=\sqrt{2(2k+\alpha+1)}\displaystyle\frac{J_{2k+\alpha+1}(cx)}{\sqrt{cx}}$.
		\begin{eqnarray*}
			||j_k^{\alpha}||_{L^2[0,1]}^2&=& 2(2k+\alpha+1)\int_0^1\left|\frac{J_{2k+\alpha+1}(cx)}{\sqrt{cx}}\right|^2dx\\
			&\leq & 2(2k+\alpha+1)\int_0^1(\frac{cx}{2})^{4k+2\alpha+2}\frac{e^{2cx}}{(\Gamma(2k+\alpha+2))^2}dx\\
			&\leq & \frac{e^{2c}}{(\Gamma(2k+\alpha+2))^2}(\frac{c}{2})^{4k+2\alpha+2}.
		\end{eqnarray*}
		By \cite{N. B}, we have $$\sqrt{2e}\left(\frac{x+1/2}{e}\right)^{x+1/2}\leq \Gamma(x+1) \leq \sqrt{2\pi}\left(\frac{x+1/2}{e}\right)^{x+1/2}$$
		Then we obtain
		
		\begin{eqnarray*}
			||\mathcal{H}_c^{\alpha}(f)||_{L^2[0,1]}^2&\leq& \sum_{k\geq n}a_k^2\frac{e^{2c}}{(\Gamma(2k+\alpha+2))^2}(\frac{c}{2})^{4k+2\alpha+2}\\
			&\leq& \frac{e^{2c}}{2}\sum_{k\geq n}a_k^2(\frac{ec}{4k+2\alpha+3})^{4k+2\alpha+2}\\
			&\leq& M\frac{e^{2c}}{2}\sum_{k\geq n}(\frac{ec}{4n+2\alpha+3})^{4k+2\alpha+2}\\
			&\leq& M\frac{e^{2c}}{2}(\frac{ec}{4n+2\alpha+3})^{4n+2\alpha+2}.
		\end{eqnarray*}
		
		Where $M=\displaystyle\sup_{k\geq 0}a_k^2.$ 
		
	\end{proof}
The figure below indicates the different steps of the eigenvalues decay rate for $\alpha = 1/2$ and different values of c . Each curve gives us an idea about the three areas of decay as well as these distributions in the interval (0.1). 
\begin{figure}[h]
	\centering
	{\includegraphics[width=9cm,height=3.5cm]{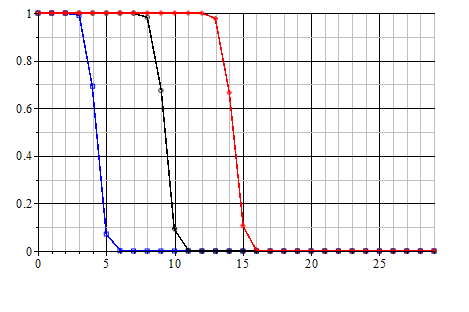}}
	\caption{(a) Graph of $\lambda_n^{(\alpha)}(c) $for $\alpha = 0.5$ and $c=$5$\pi$ (blue), $c = 10 \pi$ (black), $c = 15 \pi $ (red) } 
	
\end{figure}

\section{ The eigenvalues distribution of finite Hankel transform}
In this section, we concentrate on the distribution of $\lambda_{n}^{\alpha}(c)$ in the interval $(0,1)$ using the non-asymptotic behaviour of the trace and the Hilbert-Schmidt norm of the operator $\mathcal{Q}_c^{\alpha}.$ 
\begin{theorem}
Let $\alpha>-1/2$ and $c\geq1$, 
then we have
\begin{equation}
\mbox{trace}(\mathcal{Q}_c^{\alpha})=\frac{c}{\pi}-\frac{\alpha}{2}+\frac{\gamma_{\alpha}(c)}{c},
\end{equation}
where $|\gamma_{\alpha}(c)|\leq \gamma_{\alpha}=\gamma^1_{\alpha}+\alpha\gamma^2_{\alpha}$ and
\begin{eqnarray*}
 \gamma^1_{\alpha}&=&\left[\displaystyle\frac{(|\mu_{\alpha}|+2\alpha+1)^2-(2\alpha+3)^2+(2\alpha+1)|\mu_{\alpha}||\mu_{\alpha+1}|}{4\pi}\right]\\&+&\frac{1}{2}\left[h_{\alpha}+h_{\alpha+1}+\left(1+\sqrt{\frac{2}{\pi}}(2\alpha+1)\right)\displaystyle\left(C_{\alpha}(1+\frac{|\mu_{\alpha+1}|}{2})+C_{\alpha+1}(1+\frac{|\mu_{\alpha}|}{2})\right)\right]
\end{eqnarray*} 
$
\gamma^2_{\alpha}=\frac{\left(1+|\mu_{\alpha}|\right)\left(1+|\mu_{\alpha+1}|\right)}{\pi},$	
 $h_{\alpha}=C_{\alpha}^2+C_{\alpha}(1+\frac{|\mu_{\alpha}|}{2})$ , $\mu_{\alpha}=\alpha^2-1/4$ and $C_{\alpha}$ is given by \eqref{BoundJ_2}.
\end{theorem}
\begin{proof}
By the Mercer theorem, the trace of a kernel operator $\mathcal{H}$ is of the form
$$\mbox{trace}(\mathcal{H})=\int_{I}H(x,x)dx.$$
Then 
\begin{eqnarray*}
\mbox{trace}(\mathcal{Q}_c^{\alpha})&=&c\int_0^1G_{\alpha}(cx,cx)dx=\int_0^cG_{\alpha}(x,x)dx\\
&=&\frac{1}{2}\int_0^c(xJ_{\alpha+1}(x))'J_{\alpha}(x)dx-\frac{1}{2}\int_0^cxJ_{\alpha+1}(x)J_{\alpha}'(x)dx\\
&=&\frac{c}{2}J_{\alpha+1}(c)J_{\alpha}(c)-\int_0^cxJ_{\alpha+1}(x)J_{\alpha}'(x)dx.
\end{eqnarray*}
From \cite{Watson}, we have $xJ_{\alpha}'(x)=-xJ_{\alpha+1}(x)+\alpha J_{\alpha}(x).$ Then we obtain
\begin{equation*}
\mbox{trace}(\mathcal{Q}_c^{\alpha})=\frac{c}{2}J_{\alpha+1}(c)J_{\alpha}(c)-\alpha\int_0^cJ_{\alpha+1}(x)J_{\alpha}(x)dx+\int_0^cxJ_{\alpha+1}^2(x)dx.
\end{equation*} 
From \cite{Watson} again, we have
$$\int_0^cxJ_{\alpha+1}^2(x)dx=\frac{c^2}{2}\left(J_{\alpha+1}^2(c)-J_{\alpha}(c)J_{\alpha+2}(c)\right)$$
and
$$J_{\alpha+2}(c)=\frac{2(\alpha+1)}{c}J_{\alpha+1}(c)-J_{\alpha}(c).$$
One gets
\begin{equation*}
\mbox{trace}(\mathcal{Q}_c^{\alpha})=\frac{c^2}{2}\left(J_{\alpha+1}^2(c)+J_{\alpha}^2(c)-\frac{2\alpha+1}{c}J_{\alpha}(c)J_{\alpha+1}(c)\right)-\alpha\int_0^cJ_{\alpha+1}(x)J_{\alpha}(x)dx.
\end{equation*} 
Moreover, from \cite{NIST}, we have $\int_0^{\infty}J_{\alpha+1}(x)J_{\alpha}(x)dx=\frac{1}{2}.$ Then we obtain
\begin{equation}\label{k2}
\mbox{trace}(\mathcal{Q}_c^{\alpha})=\frac{c^2}{2}\left(J_{\alpha+1}^2(c)+J_{\alpha}^2(c)-\frac{2\alpha+1}{c}J_{\alpha}(c)J_{\alpha+1}(c)\right)-\frac{\alpha}{2}+\alpha\int_c^{\infty}J_{\alpha+1}(x)J_{\alpha}(x)dx.
\end{equation} 
By \eqref{BoundJ_2}, we have
\begin{equation}\label{eq1}
J_{\alpha}(x)=\sqrt{\frac{2}{\pi x}}r_{\alpha}(x)+\frac{\eta_{\alpha}(x)}{x^{5/2}},
\end{equation}
where $r_{\alpha}(x)=\left(\mbox{cos}(x-\omega_{\alpha})-(\alpha^2-1/4)\displaystyle\frac{\mbox{sin}(x-\omega_{\alpha})}{2x}\right)$ and $\sup_{x\geq0}|\eta_{\alpha}(x)|\leq C_{\alpha}.$\\
A straightforward computation gives us
\begin{equation}\label{eq2}
\frac{c^2}{2}\left(J_{\alpha+1}^2(c)+J_{\alpha}^2(c)-\frac{2\alpha+1}{c}J_{\alpha}(c)J_{\alpha+1}(c)\right)=\frac{c}{\pi}\left(r_{\alpha+1}^2(c)+r_{\alpha}^2(c)-\frac{2\alpha+1}{c}r_{\alpha}(c)r_{\alpha+1}(c)\right)+\frac{h_{\alpha}^{(2)}(c)}{c},
\end{equation}
where $$|h_{\alpha}^{(2)}(c)|\leq \frac{1}{2}\left[h_{\alpha}+h_{\alpha+1}+\sqrt{\frac{2}{\pi}}(2\alpha+1)\displaystyle\left(C_{\alpha}(1+\frac{|\mu_{\alpha+1}|}{2})+C_{\alpha+1}(1+\frac{|\mu_{\alpha}|}{2})\right)\right].$$
Here $h_{\alpha}=C_{\alpha}^2+C_{\alpha}(1+\frac{|\mu_{\alpha}|}{2})$ and $\mu_{\alpha}=\alpha^2-1/4.$\\
On the other hand, we have
\begin{eqnarray*}
\frac{c}{\pi}\left(r_{\alpha+1}^2(c)+r_{\alpha}^2(c)-\frac{2\alpha+1}{c}r_{\alpha}(c)r_{\alpha+1}(c)\right)&=&\frac{c}{\pi}+\frac{\mu_{\alpha}^2+2(2\alpha+1)\mu_{\alpha}-4(\alpha+1)\mbox{cos}^2(c-\omega_{\alpha})}{4\pi c}\\&+&\frac{(2\alpha+1)\mu_{\alpha}\mu_{\alpha+1}\mbox{sin}(2c-2\omega_{\alpha})}{4\pi c}.
\end{eqnarray*}

We conclude from the last equality and \eqref{eq2} that
\begin{equation}\label{k3}
\frac{c^2}{2}\left(J_{\alpha+1}^2(c)+J_{\alpha}^2(c)-\frac{2\alpha+1}{c}J_{\alpha}(c)J_{\alpha+1}(c)\right)=\frac{c}{\pi}+\frac{h_{\alpha}^{(1)}(c)+h_{\alpha}^{(2)}(c)}{c}=\frac{c}{\pi}+\frac{\gamma_{\alpha}^{(1)}(c)}{c},
\end{equation}
where $|h_{\alpha}^{(1)}(c)|\leq\displaystyle \frac{\mu_{\alpha}^2+2(2\alpha+1)|\mu_{\alpha}|+4(\alpha+1)+(2\alpha+1)|\mu_{\alpha}||\mu_{\alpha+1}|}{4\pi}.$\\
Let's proove now that $\int_c^{\infty}J_{\alpha+1}(x)J_{\alpha}(x)dx=\frac{\gamma_{\alpha}^{(2)}(c)}{c}$ where $\gamma_{\alpha}^{(2)}(c)$ is bounded. By \eqref{eq1}, we have
\begin{equation}
J_{\alpha+1}(x)J_{\alpha}(x)=\frac{2}{\pi x}r_{\alpha+1}(x)r_{\alpha}(x)+\sqrt{\frac{2}{\pi}}\frac{f_{\alpha}(x)}{x^3},
\end{equation}
where $f_{\alpha}(x)=r_{\alpha}(x)\eta_{\alpha+1}(x)+r_{\alpha+1}(x)\eta_{\alpha}(x)$ that is $|f_{\alpha}(x)|\leq C_{\alpha}(1+\frac{|\mu_{\alpha+1}|}{2})+C_{\alpha+1}(1+\frac{|\mu_{\alpha}|}{2}).$ Note that
\begin{equation}
\int_c^{\infty}J_{\alpha+1}(x)J_{\alpha}(x)dx=\frac{2}{\pi}\int_c^{\infty}\frac{r_{\alpha+1}(x)r_{\alpha}(x)}{x}dx+\sqrt{\frac{2}{\pi}}\int_c^{\infty}\frac{f_{\alpha}(x)}{x^3}dx.
\end{equation}
First, we have 
\begin{equation*}
\left|\int_c^{\infty}\frac{f_{\alpha}(x)}{x^3}dx\right|\leq \frac{C_{\alpha}(1+\frac{|\mu_{\alpha+1}|}{2})+C_{\alpha+1}(1+\frac{|\mu_{\alpha}|}{2})}{2c^2}=\frac{\gamma_{\alpha}^{(2,0)}}{c^2}.
\end{equation*}
Second, a straightforward computation gives us
\begin{equation*} 
\frac{2r_{\alpha+1}(x)r_{\alpha}(x)}{x}=\frac{\mbox{sin}(2x-2\omega_{\alpha})}{x}+\frac{\mu_{\alpha+1}\mbox{cos}^2(x-\omega_{\alpha})-\mu_{\alpha}\mbox{sin}^2(x-\omega_{\alpha})}{x^2}-\mu_{\alpha}\mu_{\alpha+1}\frac{\mbox{sin}(2x-2\omega_{\alpha})}{4x^3}.
\end{equation*}

\begin{eqnarray*}
\frac{2}{\pi}\int_c^{\infty}\frac{r_{\alpha+1}(x)r_{\alpha}(x)}{x}dx&=&\int_c^{\infty}\frac{\mbox{sin}(2x-2\omega_{\alpha})}{\pi x}dx\\&+&\int_c^{\infty}\frac{\mu_{\alpha+1}\mbox{cos}^2(x-\omega_{\alpha})-\mu_{\alpha}\mbox{sin}^2(x-\omega_{\alpha})}{\pi x^2}dx\\&-&\mu_{\alpha}\mu_{\alpha+1}\int_c^{\infty}\frac{\mbox{sin}(2x-2\omega_{\alpha})}{4\pi x^3}.
\end{eqnarray*}

\begin{equation*}
\left|\int_c^{\infty}\frac{\mbox{sin}(2x-2\omega_{\alpha})}{\pi x}dx\right|=\left|\frac{\mbox{cos}(2x-2\omega_{\alpha})}{2\pi c}-\int_c^{\infty}\frac{\mbox{cos}(2x-2\omega_{\alpha})}{2\pi x^2}dx\right|\leq\frac{1}{2\pi c}=\frac{\gamma_{\alpha}^{(2,1)}}{c}.
\end{equation*}

\begin{equation*}
\left|\int_c^{\infty}\frac{\mu_{\alpha+1}\mbox{cos}^2(x-\omega_{\alpha})-\mu_{\alpha}\mbox{sin}^2(x-\omega_{\alpha})}{\pi x^2}dx\right|\leq \frac{|\mu_{\alpha+1}|+|\mu_{\alpha}|}{\pi c}=\frac{\gamma_{\alpha}^{(2,2)}}{c}.
\end{equation*}
and
\begin{equation*}
\left|\mu_{\alpha}\mu_{\alpha+1}\int_c^{\infty}\frac{\mbox{sin}(2x-2\omega_{\alpha})}{4\pi x^3}dx\right|\leq \frac{|\mu_{\alpha+1}||\mu_{\alpha}|}{8\pi c^2}=\frac{\gamma_{\alpha}^{(2,3)}}{c^2}
\end{equation*}
Finally, we have
\begin{equation}\label{k1}
\int_c^{\infty}J_{\alpha+1}(x)J_{\alpha}(x)dx=\frac{\gamma_{\alpha}^{(2)}(c)}{c},
\end{equation}
where $|\gamma_{\alpha}^{(2)}(c)|\leq \gamma_{\alpha}^{(2,0)}+\gamma_{\alpha}^{(2,1)}+\gamma_{\alpha}^{(2,2)}+\gamma_{\alpha}^{(2,3)}=\gamma_{\alpha}^{(2)}.$
By \eqref{k1}, \eqref{k2} and \eqref{k3}, we obtain
\begin{equation}
\mbox{trace}(\mathcal{Q}_c^{\alpha})=\frac{c}{\pi}-\frac{\alpha}{2}+\frac{\gamma_{\alpha}^{(1)}(c)+\alpha\gamma_{\alpha}^{(2)}(c)}{c}.
\end{equation}
\end{proof}	
In the figure bellow, we illustrate the previous result of this work, which is given by Theorem 3. That is   $\mbox{trace}(\mathcal{Q}_c^{\alpha})-\left(\frac{c}{\pi}-\frac{\alpha}{2}\right)$ is comparable to $1/c$ with different values of $c$ and $\alpha.$

\begin{figure}[h]
	\centering
	{\includegraphics[width=14cm,height=3.5cm]{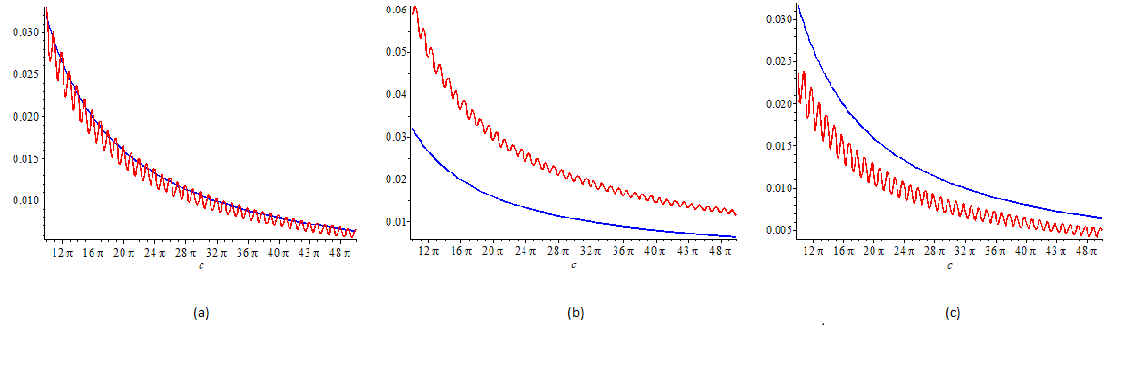}}
	\caption{(a) Graph of $\mbox{trace}(\mathcal{Q}_c^{\alpha})-(\frac{c}{\pi}-\frac{\alpha}{2}) $ (red) for $\alpha = 5/2$ (a), $\alpha = 0.55$ (b), $\alpha = 1/2$ (c) and graph of 1/c (blue)}
	
\end{figure}
 \begin{eqnarray}\label{C*}
	\mbox{Let} \hspace{1mm} C&=&\sqrt{\frac{2}{\pi}}\left[2C_{\alpha+1}(O)\left(1+\sqrt{\frac{\pi}{2}}C_{\alpha}(O)\right)+C_{\alpha}(K)\frac{(3\alpha+\frac{7}{2})C_{\alpha+2}(O)+(\alpha+\frac{3}{2})C_{\alpha}(O)}{2(\alpha+1)}\right]\\&+&\sqrt{\frac{2}{\pi}}\left[\sqrt{\frac{\pi}{2}}C_{\alpha+1}(O)\left((\frac{3}{2}+\alpha)C_{\alpha}(O)+C_{\alpha+1}(K)\right)+C_{\alpha+1}(K)+(\frac{5}{2}+\alpha)C_{\alpha+1}(O)+C_{\alpha+2}(K)\right]\nonumber.
\end{eqnarray}
\begin{theorem}
Let $\alpha>-1/2$, $c\geq1$ and $(\lambda_n^{\alpha}(c))$ be the sequence of eigenvalue of the integral operator $\mathcal{Q}_c^{\alpha}$. For every $0<\epsilon<1/2$, we have
\begin{equation}
\#\{n ; \epsilon<\lambda_n^{\alpha}(c)<1-\epsilon \}\leq \frac{16\left[\frac{\mbox{ln}(\frac{c}{\pi})}{2\pi^2}+\frac{3}{4\pi^3}+\frac{\ln(\frac{8}{3})}{\pi^2}+\frac{2c}{\pi}+1.123\right]+2C^2+4\kappa_1\mbox{ln}(2)+\frac{\kappa_2}{c^2}}{\epsilon(1-\epsilon)},
\end{equation}
where
$\kappa_1=\frac{4}{\pi^2}\left(1+(1+\sqrt{\frac{\pi}{2}}C_{\alpha}(O))^2\right)$, $\kappa_2=\frac{8C^2_{\alpha+1}(K)}{\pi^2}\left(1+(1+\sqrt{\frac{\pi}{2}}C_{\alpha}(O))^2\right)$ and $C_{\alpha}(O), C_{\alpha}(K)$ have been given by \eqref{boundJ} and \eqref{boundJ_1}.
\end{theorem}
\begin{proof}
The Hilbert-Schmidt norm of $\mathcal{Q}_c^{\alpha}$ is given by
\begin{eqnarray}
||\mathcal{Q}_c^{\alpha}||_{HS}&=&\int_0^1\int_0^1\left(K_c^{\alpha}(x,y)\right)^2dxdy\\
&=&\int_0^1\int_0^{\infty}\left(K_c^{\alpha}(x,y)\right)^2dxdy-\int_0^1\int_1^{\infty}\left(K_c^{\alpha}(x,y)\right)^2dxdy
\end{eqnarray}
The Palley-Wienner space $$\mathcal{B}_c^{\alpha}=\{ f\in L^2(0,\infty);\,\ \mbox{Support} \mathcal H^{\alpha}(f)\subset [0,c]\}$$ is a reproducing kernel Hilbert space,see \cite{N.Aronszajn}. Indeed, from the Hankel inversion theorem, for every $f\in\mathcal{B}_c^{\alpha},$  we have
$$f(x)=\int_0^c\sqrt{xy}J_{\alpha}(xy)\mathcal H^{\alpha}(f)(y)dy.$$
One gets by the Cauchy-Schwarz inequality and Parseval's theorem that, for all $x\geq 0$
$$|f(x)|\leq \sqrt{c}\hspace{0.6mm}C_{\alpha}(O)||\mathcal{H}^{\alpha}(f)||_{L^2(0,\infty)}=\sqrt{c}\hspace{0.6mm}C_{\alpha}(O)||f||_{L^2(0,\infty)}.$$ 
It then follows that 
\begin{equation}
K_c^{\alpha}(x,x)=\int_0^{\infty}\left(K_c^{\alpha}(x,y)\right)^2dy.
\end{equation}	
We conclude that
 \begin{eqnarray*}
||\mathcal{Q}_c^{\alpha}||_{HS}&=&\mbox{trace}(\mathcal{Q}_c^{\alpha})-\int_0^1\int_1^{\infty}\left(K_c^{\alpha}(x,y)\right)^2dydx\\&=&\mbox{trace}(\mathcal{Q}_c^{\alpha})-\int_0^c\int_c^{\infty}\left(G_{\alpha}(x,y)\right)^2dydx\\&=&\mbox{trace}(\mathcal{Q}_c^{\alpha})-\int_0^c\int_c^{2c}\left(G_{\alpha}(x,y)\right)^2dydx-\int_0^c\int_{2c}^{\infty}\left(G_{\alpha}(x,y)\right)^2dydx\\&=&\mbox{trace}(\mathcal{Q}_c^{\alpha})-\mathcal{I}_1(c)-\mathcal{I}_2(c).
\end{eqnarray*}	
One gets,
\begin{eqnarray}
\#\{n ; \epsilon<\lambda_n^{\alpha}(c)<1-\epsilon \}&\leq&\frac{\displaystyle\sum_{n=0}^{\infty}\lambda_n^{\alpha}(c)(1-\lambda_n^{\alpha}(c))}{\epsilon (1-\epsilon)}\\&=& \frac{\mbox{trace}(\mathcal{Q}_c^{\alpha})-||\mathcal{Q}_c^{\alpha}||_{HS}}{\epsilon(1-\epsilon)}=\frac{\mathcal{I}_1(c)+\mathcal{I}_2(c)}{\epsilon(1-\epsilon)}.
\end{eqnarray}

By \eqref{boundJ_1}, $J_{\alpha}(x)=\sqrt{\frac{2}{\pi x}}\left(\mbox{cos}(x-\omega_{\alpha})+\frac{\rho_{\alpha}(x)}{x}\right)$ with $\sup_{x\geq 0}\left|\rho_{\alpha}(x)\right|\leq \sqrt{\frac{\pi}{2}}C_{\alpha}(K)$.\\Then the kernel $G_{\alpha}$ has the following form
\begin{eqnarray}\label{eq3}
G_{\alpha}(x,y)&=&\frac{2}{\pi}\left[\frac{x\mbox{sin}(x-\omega_{\alpha})\mbox{cos}(y-\omega_{\alpha})-y\mbox{sin}(y-\omega_{\alpha})\mbox{cos}(x-\omega_{\alpha})}{x^2-y^2}\right]\\
&+&\frac{2}{\pi}\left[\frac{\frac{\rho_{\alpha}(y)}{y}x\mbox{sin}(x-\omega_{\alpha})-\frac{\rho_{\alpha}(x)}{x}y\mbox{sin}(y-\omega_{\alpha})}{x^2-y^2}\right]\nonumber\\&+&\frac{2}{\pi}\left[\frac{\rho_{\alpha+1}(x)\mbox{cos}(y-\omega_{\alpha})-\rho_{\alpha+1}(y)\mbox{cos}(x-\omega_{\alpha})}{x^2-y^2}\right]\nonumber\\&+&\frac{2}{\pi}\left[\frac{\rho_{\alpha+1}(x)\frac{\rho_{\alpha}(y)}{y}-\rho_{\alpha+1}(y)\frac{\rho_{\alpha}(x)}{x}}{x^2-y^2}\right]\nonumber.
\end{eqnarray}
Using the inequality
$\left(\sum_{k=1}^na_k\right)^2\leq n\sum_{k=1}^n(a_k)^2,$ we obtain
\begin{eqnarray}
\mathcal{I}_2(c)=\int_0^c\int_{2c}^{\infty}\left(G_{\alpha}(x,y)\right)^2dydx&\leq& 4\kappa_1\int_0^c\int_{2c}^{\infty}\frac{dydx}{(y-x)^2}+4\kappa_2\int_0^c\int_{2c}^{\infty}\frac{dydx}{(y^2-x^2)^2}\\&\leq&4\kappa_1\mbox{ln}(2)+\frac{\kappa_2}{c^2},
\end{eqnarray}
with
$\kappa_1=\frac{4}{\pi^2}\left(1+(1+\sqrt{\frac{\pi}{2}}C_{\alpha}(O))^2\right)$ and
$\kappa_2=\frac{8C^2_{\alpha+1}(K)}{\pi^2}\left(1+(1+\sqrt{\frac{\pi}{2}}C_{\alpha}(O))^2\right).$\\
For the last integral $\mathcal{I}_1(c)$, using \eqref{eq3} we have
$$G_{\alpha}(x,y)=L^1_{\alpha}(x,y)+L^2_{\alpha}(x,y),$$
where $$L^1_{\alpha}(x,y)=\frac{2}{\pi}\left[\frac{x\mbox{sin}(x-\omega_{\alpha})\mbox{cos}(y-\omega_{\alpha})-y\mbox{sin}(y-\omega_{\alpha})\mbox{cos}(x-\omega_{\alpha})}{x^2-y^2}\right]$$
and $L^2_{\alpha}(x,y)=\frac{\gamma_{\alpha}(x,y)}{c}$, where $|\gamma_{\alpha}(x,y)|\leq C$ and $C$ as given by \eqref{C*}.\\ 
Note that
\begin{eqnarray*}
\mathcal{I}_1(c)&=&\int_0^c\int_c^{2c}\left(L^1_{\alpha}(x,y)+L^2_{\alpha}(x,y)\right)^2dydx\\&\leq& 2c^2\int_0^1\int_1^{2}\left(L^1_{\alpha}(cx,cy)\right)^2dydx+2C^2\\&\leq& \frac{8}{\pi^2}\int_0^1\int_1^{2}\left[\frac{cx\mbox{sinc}(\frac{c}{\pi}(x-y))}{x+y}+\frac{\mbox{sin}(cy-\omega_{\alpha})\mbox{cos}(cx-\omega_{\alpha})}{x+y}\right]^2dydx+2C^2\\&\leq&\frac{16c^2}{\pi^2}\int_0^1\int_1^{2}\left[\mbox{sinc}(\frac{c}{\pi}(x-y))\right]^2dydx+\frac{16\mbox{ln}(4/3)}{\pi^2}+2C^2.
\end{eqnarray*}
Using the fact that $[0,1]\times[1,2]=D_1\cup D_2\cup D_3\cup D_4,$ where
$$D_1=\{(x,y),\hspace{2mm} 1/2<x<1,\hspace{2mm} -x+2<y<x+1\}$$
$$D_2=\{(x,y),\hspace{2mm} 0<x<1/2,\hspace{2mm} x+1<y<-x+2\}$$
$$D_3=\{(x,y),\hspace{2mm} 0<x<1,\hspace{2mm} y<x+1,\hspace{2mm} y<-x+2\}$$
$$D_4=\{(x,y),\hspace{2mm} 0<x<1,\hspace{2mm} y>x+1,\hspace{2mm} y>-x+2\}$$
and the following change of variables $u=x-y, t=y$ with the techniques of \cite{BK} , one gets
\begin{equation}
\mathcal{I}_1(c)\leq 16\left[\frac{\mbox{ln}(\frac{c}{\pi})}{2\pi^2}+\frac{3}{4\pi^3}+\frac{\ln(\frac{8}{3})}{\pi^2}+\frac{2c}{\pi}+1.123\right]+2C^2.
\end{equation}
\end{proof}
\section{Asymptotic Expansions for CPSWFs}
In this section, we give a brief description of the computation and the decay rate of the series expansion coefficients 
of the eigenfunctions $\varphi_{n,c}^{\alpha}$ in a generalized Laguerre functions basis of $L^2(0,\infty)$ defined in \eqref{eq0.3},
 that is for all $x\geq 0$, we have 
\begin{equation}\label{eq4}
\varphi_{n,c}^{\alpha}(x)=\sum_{k=0}^{\infty}\beta^n_k(c)\psi_{k,\alpha}^{a}(x)
\end{equation}
where $\beta^n_k(c)=\int_0^{\infty}\varphi_{n,c}^{\alpha}(x)\psi_{k,\alpha}^a(x)dx.$
\begin{lemma}
	Let $a>0$ and $\mathcal{D}_c^{\alpha}$ be the Sturm-Liouville differential operator defined in \eqref{differ_operator1}. Then for every $n\geq 0$, we have
	\begin{equation}
\mathcal{D}_c^{\alpha}(\psi_{n,\alpha}^{a})=d_{n}^{n-2}(a)\psi_{n-2,\alpha}^{a}+d_{n}^{n-1}(a)\psi_{n-1,\alpha}^{a}+d_{n}^{n}(a)\psi_{n,\alpha}^{a}+d_{n}^{n+1}(a)\psi_{n+1,\alpha}^{a}+d_{n}^{n+2}(a)\psi_{n+2,\alpha}^{a},
\end{equation}
where
$$d_{n}^{n-2}(a)=\sqrt{(n-1)n(n+\alpha-1)(n+\alpha)}$$
$$d_{n}^{n-1}(a)=-\sqrt{n(n+\alpha)}\left(\frac{c^2}{a^2}-a^2+4(2n+\alpha+1)\right)$$
$$d_{n}^{n}(a)=\alpha^2-\frac{5}{4}+(2n+\alpha+1)\left(\frac{c^2}{a^2}+a^2+3(2n+\alpha)+2\right)$$
$$d_{n}^{n+1}(a)=-\sqrt{(n+1)(n+\alpha+1)}\left(\frac{c^2}{a^2}-a^2+4(2n+\alpha+1)\right)$$
$$d_{n}^{n+2}(a)=\sqrt{(n+1)(n+2)(n+\alpha+1)(n+\alpha+2)}$$
\end{lemma}
\begin{proof}
	First, we have 
	$$\dfrac{d}{dx} \psi_{n,\alpha}^{a}(x)=(\frac{\alpha+1/2}{x}-a^2x)\psi_{n,\alpha}^{a}(x)+2\sqrt{2}a^{\alpha+3}x^{\alpha+3/2}e^{-\frac{a^2x^2}{2}}\dfrac{d}{dx}\left(\widetilde{L}_n^{\alpha}\right)(a^2x^2),$$
	then
	\begin{eqnarray*}
	\dfrac{d^2}{dx^2} \psi_{n,\alpha}^{a}(x)&=&\left[\frac{\alpha^2-1/4}{x^2}-2a^2(\alpha+1/2)+a^4x^2-a^2\right]\psi_{n,\alpha}^{a}(x)\\&+&4\sqrt{2}a^{\alpha+3}x^{\alpha+1/2}e^{-\frac{a^2x^2}{2}}\left[a^2x^2\frac{d^2\widetilde{L}_n^{\alpha}}{dx^2}(a^2x^2)+(\alpha+1-a^2x^2)\frac{d\widetilde{L}_n^{\alpha}}{dx}(a^2x^2)\right].
	\end{eqnarray*}
By \eqref{diff,Lag}, we obtain
$$\dfrac{d^2}{dx^2} \psi_{n,\alpha}^{a}(x)=\left[\frac{\alpha^2-1/4}{x^2}-2a^2(2n+\alpha+1)+a^4x^2\right]\psi_{n,\alpha}^{a}(x).$$
Finally, one gets
\begin{eqnarray}\label{M}
\mathcal{D}_c^{\alpha}(\psi_{n,\alpha}^{a})(x)&=&\mathcal{D}_c^{\alpha}(\psi_{n,\alpha}^{a})=-\dfrac{d}{dx} \left[ (1-x^2)\dfrac{d}{dx} \psi_{n,\alpha}^{a}(x) \right] - \left( \dfrac{\dfrac{1}{4}-\alpha^2}{x^2}-c^2x^2 \right)\psi_{n,\alpha}^{a}(x)\\&=&\left[(c^2-a^4)x^2+(\alpha+1)^2-1/4+2a^2(2n+\alpha+1)+2a^2x^2(2n+\alpha)+a^4x^4\right]\psi_{n,\alpha}^{a}(x)\nonumber\\&+&4\sqrt{2}a^{\alpha+3}x^{\alpha+5/2}e^{-\frac{a^2x^2}{2}}\frac{d\widetilde{L}_n^{\alpha}}{dx}(a^2x^2)\nonumber.	
\end{eqnarray}
By \cite{NIST} and using \eqref{recu.Lag}, we have $$a^2x^2\frac{d\widetilde{L}_n^{\alpha}}{dx}(a^2x^2)=(a^2x^2-(n+\alpha+1))\widetilde{L}_n^{\alpha}(a^2x^2)+\left((n+1)(n+\alpha+1)\right)^{1/2}\widetilde{L}_{n+1}^{\alpha}(a^2x^2),$$ 
$$a^2x^2\psi_{n,\alpha}^{a}(x)=-\left(n(n+\alpha)\right)^{1/2}\psi_{n-1,\alpha}^{a}(x)+(2n+\alpha+1)\psi_{n,\alpha}^{a}(x)-\left((n+1)(n+\alpha+1)\right)^{1/2}\psi_{n+1,\alpha}^{a}(x).$$
By using the two last equalities and \eqref{M}, one gets the desired result.
\end{proof}
\begin{proposition}
	Let $\big(\chi_{n,\alpha}(c)\big)$ be the sequence of eigenvalues of the differential operator $\mathcal{D}_c^{\alpha}$, then the sequence of coefficients $\big(\beta_k^n(c)\big)$ satisfy the following recurssion formula, for every $n\geq 0$ and $k\geq 0$, we have
	\begin{eqnarray}
   d_{k+2}^{k}(a)\beta_{k+2}^n(c)+d_{k+1}^{k}(a)\beta_{k+1}^n(c)+(d_{k}^{k}(a)-\chi_{n,\alpha}(c))\beta_k^n(c)+d_{k-1}^{k}(a)\beta_{k-1}^n(c)+d_{k-2}^{k}(a)\beta_{k-2}^n(c)=0,
	\end{eqnarray}
	with $d_{-1}^{1}(a)=d_{-2}^{0}(a)=d_{-1}^{0}(a)=0.$
\end{proposition} 
\begin{remark}
	The previous system can be written by the following eigensystem
	$$MD=\chi_{n,\alpha}(c)D,  M=[m_{k,j}]_{k,j\geq 0},   D=[\beta_k^n(c)]^{T}_{k\geq 0},$$
	with $m_{k,j}=0,$ if $|k-j|>2$ and $m_{k,k}=d_{k}^{k}(a)$, 
	     $m_{k,k-1}=d_{k-1}^{k}(a)$,   $m_{k,k-2}=d_{k-2}^{k}(a),$ \\
	     $m_{k,k+1}=d_{k+1}^{k}(a)$ ,$m_{k,k+2}=d_{k+2}^{k}(a).$
	     Moreover, for $a=\sqrt{c},$ the previous matrix becomes diagonally dominant when $n<\frac{1}{2}(\frac{c}{4}-(\alpha+1))$ where we can use the Inverse Power Method for the computation of CPSWFs when c is large compared to the prolate's order.
\end{remark}
\begin{proposition}
Let $c > 0$ be a positive real number, then there exists a constant $\delta_0$ and a positive integer $k_0$ such that for any integer $n\geq \max{(\frac{c}{2},\frac{c}{\pi}+k_0)}$, $\chi_{n,\alpha}(c)> \max{(2\alpha^2-1/2,c^2(4\alpha^2-1))}$ and $k\geq n$ , we have
\begin{equation}
|\beta^n_k(c)|\leq \frac{1}{\sqrt{2e\delta_0}}\dfrac{M_{\alpha}(k)}{(\alpha+1)}(\frac{c}{a})^{\alpha+1}e^{-(k+\alpha+1/2)\ln(\frac{k+\alpha+1/2}{2})+\frac{A}{2}(2n+\alpha+1)\ln(\frac{\pi}{c}(n+k_0))}.
\end{equation}
Where $M_{\alpha}(k)=\begin{cases}\frac{(\alpha+1)_k}{k!}& \mbox{if }  \alpha\geq 0\\(2-\frac{(\alpha+1)_k}{k!})& \mbox{if }  -1\leq\alpha\leq 0\end{cases}$.
\end{proposition}
	\begin{proof}
		First, from the Parseval's equality, \eqref{eq0.5} and \eqref{eq0.3}, we have
		\begin{eqnarray*}
	\beta^n_k(c)&=&\left <\varphi_{n,c}^{\alpha},\psi_{k,\alpha}^a\right >_{L^2(0,\infty)}=\left <\mathcal{H}^{\alpha}(\varphi_{n,c}^{\alpha}),\mathcal{H}^{\alpha}(\psi_{k,\alpha}^a)\right >_{L^2(0,\infty)}\\&=&\frac{(-1)^k}{ac\mu_{n,\alpha}(c)}\left <\varphi_{n,c}^{\alpha}(\frac{.}{c})\chi_{[0,c]}, \psi_{k,\alpha}^a(\frac{.}{a^2})\right >_{L^2(0,\infty)}\\&=&\frac{(-1)^k}{a\mu_{n,\alpha}(c)}\left <\varphi_{n,c}^{\alpha}, \psi_{k,\alpha}^a(\frac{c}{a^2}.)\right >_{L^2(0,1)}
	\end{eqnarray*}
Then, by the Cauchy-Schwarz inequality, one gets
\begin{eqnarray*}
	|\beta^n_k(c)|\leq \frac{1}{a|\mu_{n,\alpha}(c)|} ||\psi_{k,\alpha}^a(\frac{c}{a^2}.)||_{L^2(0,1)}.
	\end{eqnarray*}	
By \cite{Michalska-Szynal}, we have 
$$||\psi_{k,\alpha}^a(\frac{c}{a^2}.)||_{L^2(0,1)}\leq \dfrac{c^{\alpha+1/2}M_{\alpha}(k)}{a^{\alpha}(\alpha+1)\Gamma(k+\alpha+1)},$$
where $M_{\alpha}(k)=\begin{cases}\frac{(\alpha+1)_k}{k!}& \mbox{if }  \alpha\geq 0\\(2-\frac{(\alpha+1)_k}{k!})& \mbox{if }  -1\leq\alpha\leq 0\end{cases}$.
Finally, from \cite{N. B}, \cite{Karoui-Boulsane} and the last inequality, one gets
\begin{eqnarray*}
	|\beta^n_k(c)|&\leq& \dfrac{c^{\alpha+1/2}M_{\alpha}(k)}{a^{\alpha+1}(\alpha+1)\Gamma(k+\alpha+1)}\frac{1}{|\mu_{n,\alpha}(c)|}\\&\leq& \dfrac{c^{\alpha+1/2}M_{\alpha}(k)}{\sqrt{2e}a^{\alpha+1}(\alpha+1)}\left(\frac{e}{k+\alpha+1/2}\right)^{k+\alpha+1/2}\sqrt{\frac{c}{\delta_0}}e^{\frac{A}{2}(2n+\alpha+1)\ln(\frac{\pi}{c}(n+k_0))}\\&\leq& \dfrac{c^{\alpha+1}M_{\alpha}(k)}{\sqrt{2e\delta_0}a^{\alpha+1}(\alpha+1)}e^{-(k+\alpha+1/2)\ln(\frac{k+\alpha+1/2}{2})+\frac{A}{2}(2n+\alpha+1)\ln(\frac{\pi}{c}(n+k_0))}.
	\end{eqnarray*}
\end{proof}


\begin{thebibliography}{99}
\bibitem{Abreu} L. D. Abreu and A. S. Bandeira, Landau's necessary conditions for the Hankel transform,{ J. Funct. Anal.} {\bf 262}(4), (2012),  1845--1866.
	
\bibitem{Andrews} G. E. Andrews, R. Askey and  R. Roy, Special Functions,
Cambridge University Press , Cambridge, New York, 1999.

\bibitem{N.Aronszajn} N.Aronszajn, Theory of reproducing Kernels, American Mathematical Society, Vol. 68, No. 3 (May, 1950), pp. 337-404.


\bibitem{N. B} N. Batir, Inequalities for the gamma function, Arch. Math. 2008; 91(6): 554–56

\bibitem{BJK} A. Bonami, P. Jaming  and A. Karoui, Non-Asymptotic Behaviour of the Sinc-Kernel Operator and Related Applications,
available at arXiv:1804.01257, (2018).

\bibitem{BK} A. Bonami and A. Karoui, Random Discretization of the Finite Fourier Transform and Related Kernel Random Matrices,
available at arXiv:1703.10459, (2019).

\bibitem{Karoui-Boulsane} M. Boulsane and A. Karoui, The Finite Hankel Transform Operator: Some Explicit and Local Estimates of the Eigenfunctions and Eigenvalues Decay Rates, J. Four. Anal. Appl, Volume 24, Issue 6, pp 1554–1578, (2018).
	
\bibitem{L.G} L. Gatteschi, Asymptotics and bounds for the zeros of Laguerre polynomials: a survey. Journal of Computational and Applied Mathematics, 144(1): 7-27, (2002).

\bibitem{Ilia Krasikov} I. Krasikov, Approximation for the Bessel and airy functions with an explicit error term, LMS Journal of Computation and Mathematics,  Volume 17, Issue 1 pp. 209-225, (2014).

\bibitem{Michalska-Szynal} M.Michalska and J.Szynal, A new bound for the Laguerre polynomials, Journal of Computational and Applied Mathematics 133(1-2):489–493, (2001).

\bibitem{Slepian3}  D.Slepian, Prolate spheroidal wave functions, Fourier analysis and uncertainty--IV: Extensions to many dimensions; generalized
prolate spheroidal functions, {\it Bell System Tech. J.} {\bf 43}
(1964), 3009--3057.

\bibitem{A.YA. OLenko} A.YA. Olenko, Upper bound on $\sqrt{x}J_{\mu}(x)$ and its applications, Integral Transforms and Special Functions.Vol. 17, No. 6, June 2006, 455–467
 
 \bibitem{NIST} F.W.J. Olver,  D.W. Lozier,   R.F.Boisvert and C.W.Clark, NIST Handbook of Mathematical Functions,  Cambridge University Press; New York;  (2010).
 
 \bibitem{Xiao-Rokhlin} H. Xiao and V. Rokhlin, High-Frequency Asymptotic	Expansions for Certain Prolate Spheroidal Wave Functions,
 J.Four. Anal. Appl.Volume 9, Issue 6, (2003).

\bibitem{Watson} G.\,N. Watson, A treatise on the theory of Bessel functions.second edition, Cambridge University Press.(1966).
\end{thebibliography}
\end{document}